\documentclass{amsart}

\usepackage{a4wide,latexsym,amssymb,amsthm,amsmath,color}

\theoremstyle{plain}

\newtheorem{theorem}{Theorem}
\newtheorem{lemma}{Lemma}

\newtheorem{definition}{Definition}

\theoremstyle{remark}

\newtheorem{remark}{Remark}
\newtheorem{example}{Example}

\def\E{\mathbb{E}}

\begin{document}

\title[surfaces in a pseudo-sphere with harmonic or 1-type pseudo-spherical Gauss map]{Surfaces in a pseudo-sphere with harmonic or 1-type \\ pseudo-spherical Gauss map}

\author[B. Bekta\c s]{Burcu Bekta\c s}
\address{Istanbul Technical University, Faculty of Science and Letters, Department of Mathematics, 34469, Maslak, 
Istanbul, Turkey} 
\email{bektasbu@itu.edu.tr}

\author[J. Van der Veken]{Joeri Van der Veken}
\address{KU Leuven, Department of Mathematics, Celestijnenlaan 200B -- Box 2400, BE-3001 Leuven, Belgium} 
\email{joeri.vanderveken@wis.kuleuven.be}

\author[L. Vrancken]{Luc Vrancken}
\address{LAMAV, ISTV2 \\Universit\'e de Valenciennes\\ Campus du Mont Houy\\ 59313 Valenciennes Cedex 9\\ France and KU  Leuven \\ Department of Mathematics \\Celestijnenlaan 200B -- Box 2400 \\ BE-3001 Leuven \\ Belgium}  
\email{luc.vrancken@univ-valenciennes.fr}

\thanks{This work is partially supported by the Belgian Interuniversity Attraction Pole P07/18 (Dygest) and was carried out while the first author visited KU Leuven supported 
by The Scientific and Technological Research Council of Turkey (TUBITAK) under grant 1059B141500244.}

\begin{abstract}
We give a complete classification of Riemannian and Lorentzian surfaces of arbitrary codimension in a pseudo-sphere whose pseudo-spherical Gauss maps are of 1-type or, in particular, harmonic. In some cases a concrete global classification is obtained, while in other cases the solutions are described by an explicit system of partial differential equations.
\end{abstract}



\maketitle

%%%%%%%%%%%%%%%%%%%%%%%%%%%%%%%%%%%%%%%%%%%%%%%%%%%%%%%%%%%%%%%%%%%%%%%%%%%%%%%%%%%%%%%%%%%%%%%%%%%%%%%%%%%%%%%%%%%%%%%

\section{Introduction}
In the late 1970's, B.-Y. Chen introduced the concept of finite type submanifolds in Euclidean space. Since then, finite type theory became an active research field of which the first results and fundamental notions were collected in the book \cite{Chenbook1}. Later, the definition of finite type submanifolds was extended to differentiable maps on Riemannian manifolds, in particular to Gauss maps of Euclidean submanifolds, which are useful tools in their study, see \cite{CP}. The generalization to more general maps attracted the interest of people working in analysis to finite type theory and it is now considered an important field with several open conjectures. For a current state of the art, see the report \cite{C2} and the very recent second edition of Chen's above mentioned book \cite{Chenbook2}.

A smooth map $\phi:M\longrightarrow\mathbb{E}^m$ from a Riemannian manifold $M$ into a Euclidean space $\mathbb{E}^m$ is said to be \emph{of finite type} if it has a finite spectral decomposition, i.e., if it can be written as
\begin{equation*}
\phi=\phi_1+\phi_2+\ldots+\phi_k,
\end{equation*}
where each $\phi_i$ satisfies $\Delta\phi_i=\lambda_i\phi_i$ for a constant $\lambda_i\in\mathbb{R}$, where $\Delta$ is the Laplacian of $M$ acting on each component of $\phi_i$. If $\lambda_1, \lambda_2, \dots, \lambda_k$ are all distinct, $\phi$ is said to be \emph{of $k$-type}. 
%Remark that $\phi$ is \emph{harmonic} if it is of $1$-type, with $\lambda_1=0$.

Let $\textbf{x}:M \longrightarrow\mathbb{E}^m$ be an isometric immersion of an oriented $n$-dimensional Riemannian manifold $M$ into a Euclidean space $\mathbb{E}^m$. Let $G(n,m)$ denote the Grassmannian manifold consisting of all oriented $n$-planes through the origin  of $\mathbb{E}^m$. The classical Gauss map $\nu: M\longrightarrow G(n,m)$ is a smooth map which carries each point $p\in M$ to the oriented $n$-plane in $\mathbb E^m$ obtained by parallel translation of the tangent space to $\mathbf{x}(M)$ at $\mathbf{x}(p)$ in $\mathbb{E}^m$ to the origin. Note that since $G(n,m)$ is canonically embedded into a Euclidean space $\mathbb{E}^N$, where $N= {m\choose n}$, the notion of finite type map can be defined for the classical Gauss map. The classical Gauss map associated with a pseudo-Riemannian submanifold of a pseudo-Euclidean space was given in a similar way in \cite{Kim}. 

Chen and Piccinni gave a characterization theorem for submanifolds of $\mathbb{E}^m$ with $1$-type Gauss map in \cite{CP}, a result which is closely related to the well-known characterization of parallel mean curvature submanfolds of $\mathbb{E}^m$ by the harmonicity of their Gauss map $\nu: M\longrightarrow G(n,m)$ by Ruh and Vilms given in \cite{RV}. 

An isometric immersion of an $n$-dimensional Riemannian manifold $M$ into a sphere $\mathbb S^{m}$ can also be seen as an isometric immersion into a Euclidean space $\mathbb{E}^{m+1}$, and therefore the Gauss map associated with such an immersion can be determined in the classical sense as above. On the other hand, Obata modified the definition of the Gauss map to better capture the properties of the immersion into the sphere, rather than into the Euclidean space, in \cite{O} as follows. Let $\textbf{x}: M \longrightarrow\widetilde M$ be an isometric immersion from an $n$-dimensional Riemannian manifold $M$ into an $m$-dimensional real space form $\widetilde M$. The generalized Gauss map in Obata's sense is a map which assigns to each $p\in M$ the totally geodesic $n$-dimensional submanifold of $\widetilde M$ tangent to $\textbf{x}(M)$ at $\textbf{x}(p)$. In the case $\widetilde M = \mathbb  S^{m}$, Obata's map carries $p\in M$ to a totally geodesic $n$-sphere of $\mathbb{S}^m$ and this is of course uniquely determined by a linear $(n+1)$-dimensional subspace of $\mathbb{E}^{m+1}$, by intersecting this subspace with $\mathbb S^m$. Hence, this map can be seen as a map from $M$ to $G(n+1,m+1)$, called the \emph{spherical Gauss map}. 

For Riemannian submanifolds of a sphere with 1-type or harmonic spherical Gauss map, Chen and Lue gave a characterization theorem in \cite{CL}.

In \cite{Beu}, the first author, Canfes and Dursun introduced the notion of \emph{pseudo-spherical Gauss map} associated with an immersion of a (pseudo-)Riemannian manifold into a pseudo-sphere and also obtained some characterization and classification theorems. Note however that, in \cite{Ish}, Ishihara studied the Gauss map in a generalized sense for (pseudo-)Riemannian submanifolds of (pseudo-) Riemannian manifolds, also extending the Gauss map in Obata's sense to the pseudo-Riemannian setting. Some of the classification results in \cite{Beu} deal with Lorentzian surfaces in $\mathbb S^4_1$ and $\mathbb S^4_2$ with harmonic pseudo-spherical Gauss map and with Riemannian and Lorentzian surfaces in $\mathbb S^4_1$ with $1$-type pseudo-spherical Gauss map. In this paper we unify and extend these results to arbitrary codimension and index, i.e., we completely classify all Riemannian and Lorentzian surfaces in $\mathbb S^m_s$ with 1-type Gauss map and in particular with harmonic Gauss map. The main theorems are Theorem 3, 4, 5 and 6.

\section{Preliminaries}

\subsection{Pseudo-Riemannian real space forms}

Let $\mathbb{E}^m_s$ denote the \emph{pseudo-Euclidean space} of dimension $m$ and index $s$, i.e., $\mathbb R^m = \{ (x_1,\ldots,x_m) \ | \ x_1,\ldots,x_m \in \mathbb R \}$ equipped with the metric
\begin{equation} 
\label{metric}
ds^2 = \sum_{i=1}^{m-s} dx_i^2-\sum_{j=m-s+1}^{m} dx_j^2.
\end{equation}
Then $\mathbb{E}^m_s$ has constant sectional curvature $c=0$. If $\langle\cdot,\cdot\rangle$ is the inner product associated with $ds^2$, we define for any real number $c \neq 0$ 
\begin{align*} 
& \mathbb{S}^m_s(c) = \left\{x\in\mathbb{E}_s^{m+1} \ |\  \langle x,x \rangle = c^{-1}\right\} 
\mbox{ if } c>0, \\
& \mathbb{H}^m_s(c) = \left\{x\in\mathbb{E}_{s+1}^{m+1} \ |\ \langle x,x \rangle =c^{-1}\right\}
\mbox{ if } c<0. 
\end{align*}
When equipped with the induced metric from $ds^2$, these manifolds have constant sectional curvature $c$ and are called \emph{pseudo-sphere}, respectively \emph{pseudo-hyperbolic space}, of dimension $m$ and index $s$. A vector $v$ tangent to one of these spaces is called \emph{spacelike} if $\langle v, v\rangle>0$ or $v=0$, \emph{timelike} if $\langle v,v\rangle<0$, and \emph{lightlike} (or \emph{null}) if $\langle v, v\rangle=0$ and $v\neq 0$. This terminology is inspired by general relativity, where $\mathbb{E}^m_1$, $\mathbb{S}^{m-1}_1(c)$ and $\mathbb{H}^{m-1}_1(c)$ are known as the Minkowski, de Sitter, and anti-de Sitter spaces, respectively.

\subsection{Basics of submanifold theory}

Let $M$ be an $n$-dimensional (pseudo-)Riemannian submanifold of a (pseudo-)Riemannian manifold $\widetilde M$. We denote the Levi-Civita connections of $\widetilde M$ and $M$ by $\widetilde{\nabla}$ and $\nabla$ respectively. Then the Gauss and Weingarten formulas are given respectively by 
\begin{align*}
& \widetilde{\nabla}_{X}Y = \nabla_{X}Y + h(X,Y), \\
& \widetilde{\nabla}_{X}\xi =-A_{\xi}X+ D_{X}\xi
\end{align*}
for any vector fields $X$ and $Y$ tangent to $M$ and any vector field $\xi$ normal to $M$. Here, $h$ is the second fundamental form, a symmetric tensor field taking values in the normal bundle, $A_{\xi}$ stands for the shape operator with respect to the normal direction $\xi$ and $D$ is a connection in the normal bundle. The shape operators and the second fundamental form are related by 
\begin{equation}
\label{eq:1}
\langle A_\xi X,Y \rangle=\langle h(X,Y), \xi\rangle
\end{equation}
for any $X$ and $Y$ tangent to $M$ and any $\xi$ normal to $M$, where $\langle \cdot,\cdot \rangle$ denotes the metric both on $M$ and on $\widetilde M$. The mean curvature vector field is defined as 
$$H=\frac{1}{n} \, \mbox{tr}\,{h}.$$

If the ambient manifold $\widetilde M$ has constant sectional curvature $c$, the equations of Gauss, Codazzi and Ricci are given respectively by 
\begin{align}
\label{Gausseq}
& \langle R(X,Y)Z,W\rangle=
c(\langle X,W\rangle\langle Y,Z\rangle \! - \! \langle X,Z\rangle\langle Y,W\rangle) \!+\! \langle h(X,W), h(Y,Z)\rangle \!-\! \langle h(Y,W), h(X,Z)\rangle,\\
\label{Codazzieq}
& (\overline{\nabla} h)(X,Y,Z)=(\overline{\nabla} h)(Y,X,Z),\\
\label{Riccieq}
& \langle R^D(X,Y)\xi,\eta\rangle=\langle[A_\xi,A_\eta]X,Y\rangle
\end{align}
for any $X$, $Y$, $Z$ and $W$ tangent to $M$ and any $\xi$ and $
\eta$ normal to $M$. Here, $R$ is the curvature tensor of $M$, $R^D$ is the curvature tensor associated with $D$ and the covariant derivative $\overline{\nabla} h$ is defined by 
\begin{align}
(\overline{\nabla} h)(X,Y,Z)&=D_Xh(Y,Z)-h(\nabla_X Y,Z)-h(Y,\nabla_X Z)
\end{align}
for any $X$, $Y$ and $Z$ tangent to $M$. If $\{e_1,\ldots,e_n\}$ is a local frame on $M$ satisfying 
$\langle e_i,e_j \rangle = \varepsilon_i \delta_{ij}$, with $\varepsilon_i \in \{-1,1\}$ for all $i$ and $j$, the scalar curvature of $M$ is given by
$$ S = \sum_{i,j=1}^n \varepsilon_i \varepsilon_j \langle R(e_i,e_j)e_j,e_i \rangle. $$
In particular, if $n=2$, the scalar curvature equals twice the Gaussian curvature $K$ of the surface~$M$. It follows from the equation of Gauss \eqref{Gausseq} that 
\begin{equation} 
\label{scalar-curv-sphere}
S = cn(n-1) + n^2 \langle H,H \rangle - S_h,
\end{equation}
where 
$$ S_h = \sum_{i,j=1}^n \varepsilon_i \varepsilon_j \langle h(e_i,e_j),h(e_i,e_j) \rangle. $$

\subsection{Finite type maps}

Following \cite{C}, we define finite type maps
from a pseudo-Riemannian manifold to a pseudo-sphere or a pseudo-hyperbolic space as follows.
\begin{definition}
\label{finiteteypedef}
A smooth map $\phi:M\longrightarrow\mathbb{S}^{m}_s(c)\subset\mathbb{E}^{m+1}_s$ (resp. $\phi:M\longrightarrow\mathbb{H}^{m}_{s}(c)\subset\mathbb{E}^{m+1}_{s+1}$), from a (pseudo-)Riemannian manifold into a pseudo-sphere (resp. a pseudo-hyperbolic space) is called \emph{of finite type in $\mathbb{S}^{m}_s(c)$ (resp. in $\mathbb{H}^{m}_{s}(c)$)} if it has a finite spectral decomposition
\begin{equation}
\label{specde}
\phi=\phi_1+\phi_2+\cdots+\phi_k, 
\end{equation}
where each $\phi_i$ satisfies $\Delta\phi_i=\lambda_i \phi_i$ for some constant $\lambda_i\in\mathbb{R}$, where $\Delta$ is the Laplacian of $M$ acting on each of the $m+1$ component functions of $\phi_i$. If the spectral decomposition \eqref{specde} contains exactly $k$ terms with different values for $\lambda_i$, then the map $\phi$ is called \emph{of $k$-type}.
\end{definition}

\begin{remark} \label{rem1}
In particular, $\phi:M\longrightarrow\mathbb{S}^{m}_s(c)\subset\mathbb{E}^{m+1}_s$ 
(resp. $\phi:M\longrightarrow\mathbb{H}^{m}_{s}(c)\subset\mathbb{E}^{m+1}_{s+1}$) is of 1-type in 
$\mathbb{S}^{m}_s(c)$ (resp. in $\mathbb{H}^{m}_{s}(c)$) if and only if $\Delta\phi=\lambda \phi$ 
for some $\lambda \in \mathbb R$. As is well-known, if $\lambda = 0$, the map $\phi$ is a harmonic map.
More about harmonic maps can be found in Remark \ref{rem_harm} below.
\end{remark}

\subsection{The pseudo-spherical Gauss map}

Let us first recall some basics about Grassmannian manifolds. If $G(n+1,m+1)$ is the set of all oriented non-degenerate $(n+1)$-dimensional linear subspaces of $\mathbb E^{m+1}_s$, we can construct a natural inclusion of $G(n+1,m+1)$ into $\bigwedge^{n+1}\E^{m+1}_s$ by identifying $L \in G(n+1,m+1)$ with $e_0 \wedge \ldots \wedge e_n$, where $\{e_0,\ldots,e_n\}$ is a positively oriented orthonormal basis for $L$. Furthermore, we can identify $\bigwedge^{n+1}\mathbb E^{m+1}_{s}$ with a pseudo-Euclidean space $\mathbb E^{N}_q$ for some non-negative integer $q$ and $N= {m+1 \choose n+1}$. Here, the pseudo-Euclidean metric is defined as follows: if $\{f_1, f_2, \ldots, f_{m+1}\}$ and $\{g_1, g_2, \ldots, g_{m+1}\}$ are orthonormal bases of $\mathbb E^{m+1}_{s}$, the inner product of $f_{i_1} \wedge \ldots \wedge f_{i_{n+1}}$ and $g_{j_1} \wedge \ldots \wedge g_{j_{n+1}}$ in $\bigwedge^{n+1} \mathbb E^{m+1}_{s}$ is given by
\begin{equation} \label{inner-prod}
\left\langle \left\langle f_{i_1} \wedge \ldots \wedge f_{i_{n+1}},  g_{j_1} \wedge \ldots \wedge g_{j_{n+1}}\right\rangle \right\rangle
= \det( \left\langle f_{i_\ell}, g_{j_k}\right\rangle).
\end{equation} 
In this way, $G(n+1,m+1)$ can be seen as a submanifold of $\mathbb E^{N}_q$. 
%\red{Remark that the dimension of $G(n+1,m+1)$ is $(n+1)(m-n)$.}

Now, let $\mathbf{x}: M^n_t \longrightarrow \mathbb{S}^{m}_s(1)\subset\mathbb{E}^{m+1}_s$ be an isometric immersion from an oriented $n$-dimensional (pseudo-)Riemannian manifold of index $t$ into a pseudo-sphere. As usual, we will locally identify $M^n_t$ with $\mathbf{x}(M^n_t)$. The \emph{pseudo-spherical Gauss map} $\nu$ in the sense of Obata \cite{O}, assigns to each point $p\in M^n_t$ the great pseudo-subsphere $\mathbb{S}^{n}_{t}(1)$ of $\mathbb{S}^{m}_{s}(1)$ tangent to $M^n_t$ at $p$. Since $\mathbb{S}^{n}_{t}(1)$ is uniquely determined as the intersection of $\mathbb{S}^{m}_{s}(1)$ with an $(n+1)$-plane through the center of $\mathbb{S}^{m}_{s}(1)$ in $\mathbb{E}^{m+1}_{s}$, one can look at $\nu$ as a map from $M^n_t$ to $G(n+1,m+1)$. In particular, if $\{e_1,\ldots,e_n\}$ is a positively oriented orthonormal basis of $T_pM^n_t$, then 
$$ \nu(p) = \mathbf{x}(p) \wedge e_1 \wedge \ldots \wedge e_n \in G(n+1,m+1) \subset \mathbb E^N_q. $$

From \eqref{inner-prod} we see $\left\langle  \left\langle \nu,\nu \right\rangle \right\rangle = (-1)^{t}$, 
and hence $\nu$ is a map from $M^n_t$ to $\mathbb S^{N-1}_q(1)$ if $t$ is even and 
a map from $M^n_t$ to $\mathbb H^{N-1}_{q-1}(-1)$ if $t$ is odd. 
It is therefore natural to investigate for which submanifolds $M_t$ of $\mathbb S^m_s(1)$ the pseudo-spherical Gauss map is of finite type in the sense of Definition \ref{finiteteypedef}. The following results were obtained in \cite{Beu}.

\begin{lemma}
\cite{Beu} \label{calc-laplce-lem}
Let $M^n_t$ be an $n$-dimensional oriented (pseudo-)Riemannian submanifold of index~$t$ of $\mathbb{S}^{m}_s(1)$. Then the Laplacian of the pseudo-spherical Gauss map $\nu: M^n_t\rightarrow G(n+1, m+1)\subset\mathbb{E}^N_q$ is given by
\begin{align} \label{deltanu}
\begin{split}
\Delta \nu = &  S_h \nu + n H  \wedge e_{1} 
\wedge  \cdots \wedge e_{n} 
- n \sum_{k=1}^n   {\bf x} \wedge  e_{1} \wedge \cdots \wedge  
\underbrace{D_{e_k} H }_{k-th} \wedge \cdots \wedge e_{n}  \\ 
& +\sum_{j,k=1 \atop j\neq k}^n\sum_{r,s=n+1,\atop r<s}^{m} \varepsilon_r \varepsilon_s R^r_{sjk} \,
{\bf x} \wedge e_{1} \wedge \cdots 
\wedge  \underbrace{e_r}_{j-th}  \wedge \cdots \wedge  \underbrace{e_s}_{k-th} \wedge \cdots \wedge e_{n},
\end{split}
\end{align}
where $\{e_1,\ldots,e_n\}$ is a local positively oriented orthonormal tangent frame to $M^n_t$, $\{e_{n+1},\ldots,e_m\}$ is a local orthonormal normal frame to $M^n_t$ in $\mathbb S^{m}_s(1)$, satisfying $\langle e_r,e_s \rangle = \varepsilon_r \delta_{rs}$, with $\varepsilon_r \in \{-1,1\}$, and $R^r_{sjk}= \langle R^D(e_j, e_k) e_r, e_s \rangle$.
\end{lemma}

The following results are direct consequences of Remark \ref{rem1}, Lemma \ref{calc-laplce-lem} and 
formula \eqref{scalar-curv-sphere}.

\begin{theorem}
\cite{Beu}
\label{1typesubmanifold}
A (pseudo-)Riemannian submanifold $M$ of $\mathbb{S}^{m}_s(1)$ has 1-type pseudo-spherical Gauss map if and only if $M$ has zero mean curvature vector field and flat normal connection in $\mathbb{S}^{m}_s(1)$, and constant scalar curvature. 
\end{theorem}

\begin{theorem} 
\cite{Beu}
\label{prop-1}
An $n$-dimensional (pseudo-)Riemannian submanifold $M$ of $\mathbb{S}^{m}_s(1)$ has harmonic pseudo-spherical Gauss map 
if and only if $M$ has zero mean curvature vector field and flat normal connection in $\mathbb{S}^{m}_s(1)$, 
and constant scalar curvature $S=n(n-1)$.
\end{theorem}

\begin{remark} \label{rem_harm}
In Theorem \ref{prop-1}, harmonicity of $\nu$ is interpreted as the vanishing of $\Delta\nu$ as computed in~\eqref{deltanu}. This means that $\nu$ is harmonic as a map between the (pseudo-)Riemannian manifolds $M$ and~$\mathbb E^N_q$. However, it is also natural to study when $\nu$ is harmonic as a map from $M$ to $G(n+1,m+1)$ (with the induced metric from $\mathbb E^N_q$). This happens if and only if the terms tangent to $G(n+1,m+1)$ in \eqref{deltanu} vanish, i.e., if and only if $M$ has zero mean curvature vector in $\mathbb S^m_s(1)$.
\end{remark}

\section{Classification results for Riemannian surfaces}

Consider a spacelike (i.e. Riemannian) surface with arbitrary codimension in a pseudo-sphere $\mathbb S^m_s(1)$. For the Riemannian case $s=0$, we refer to \cite{CL}. The next two results give a classification of those surfaces with $1$-type pseudo-spherical Gauss map, first for the non-harmonic case, then for the harmonic case. 

\begin{theorem}
\label{1typessurfacearbit}
A spacelike surface $M$ in $\mathbb{S}^{m}_s(1)$ has non-harmonic 1-type pseudo-spherical Gauss map if and only if it is an open part of the Clifford torus lying fully in a totally geodesic 3-sphere 
$\mathbb{S}^3(1)\subset\mathbb{S}^{m}_s(1)$.
\end{theorem}

\begin{proof}
Assume that $\mathbf{x}: M \longrightarrow \mathbb{S}^{m}_s(1)$ is an immersion of a spacelike surface with 1-type pseudo-spherical Gauss map. It follows from Theorem \ref{1typesubmanifold} that $M$ has zero mean curvature vector, flat normal connection in $\mathbb{S}^{m}_s(1)$ and constant scalar curvature. The equation of Ricci \eqref{Riccieq} implies that $[A_\xi,A_\eta]=0$ for any normal vector fields $\xi$ and $\eta$. Since $M$ has a positive definite metric, $A_\xi$ and $A_\eta$ are simultaneously diagonalizable. Let us choose $\{e_1,e_2\}$ to be a local orthonormal tangent frame such that $h(e_1,e_2)=0$ and $\{e_3, \dots, e_{m}, e_{m+1}={\bf x}\}$ to be a local orthonormal normal frame along $M$ in $\mathbb{E}^{m+1}_s$. The connection form $\omega_{12}$ on $M$ is defined as usual by $\nabla_Xe_1=\omega_{12}(X)e_2$ for any $X$ tangent to $M$. Since the mean curvature vector field is zero, we have $h(e_1,e_1)=-h(e_2,e_2)$. Putting $\xi=h(e_1,e_1)$, it follows from the Gauss equation \eqref{Gausseq} that the constant Gaussian curvature of $M$ is $K=1-\langle\xi,\xi\rangle$. Hence, $\langle\xi,\xi\rangle$ is constant and there are three possibilities according to the casual character of $\xi$.

\emph{Case (i).} \emph{$\xi \neq 0$ is spacelike.} 
We can choose a unit spacelike vector field $\xi'$ such that $\xi=\lambda\xi' \neq 0$. Since $K=1-\lambda^2$ is constant, $\lambda$ is a nonzero constant. For $X=e_1$ and $Y=Z=e_2$, the Codazzi equation \eqref{Codazzieq} gives $D_{e_1}\xi'=-2\omega_{12}(e_2)\xi'$. Similarly, for $X=Z=e_1$ and $Y=e_2$ we have $D_{e_2}\xi'=2\omega_{12}(e_1)\xi'$. On the other hand, since $\|\xi'\|=1$, $\langle D_{e_i}\xi',\xi'\rangle=0$ for $i=1,2$. Thus, $\omega_{12}(e_1)=\omega_{12}(e_2)=0$, which implies that $K=0$ and the first normal space $\mbox{Im}\,h=\mbox{Span}\{\xi'\}$ is parallel in the normal bundle. From the reduction theorem by Erbacher-Magid \cite{erb,M}, $M$ is contained in a totally geodesic $\mathbb{S}^3(1)$ in $\mathbb{S}^{m}_s(1)$. It follow for example from \cite{La} that a non-totally geodesic maximal surface of $\mathbb{S}^3(1)$ with constant Gaussian curvature is an open part of the Clifford torus.  
 
\emph{Case (ii).} \emph{$\xi$ is timelike.} 
Now choose a unit timelike vector $\xi'$ with $\xi=\lambda\xi' \neq 0$, then $K=1+\lambda^2$. On the other hand, by a similar calculation as in Case \emph{(i)}, it can be shown that $K=0$. This is a contradiction.  

\emph{Case (iii).} \emph{$\xi=0$ or $\xi$ is lightlike.} 
Now $K=1$, that is $S=2$. From Theorem \ref{prop-1}, the pseudo-spherical Gauss map is harmonic, which is a contradiction.

Conversely, suppose that $M$ is an open part the Clifford torus in $\mathbb{S}^3(1)\subset\mathbb{S}^m_s(1)$. Then it is easy to show that the pseudo-spherical Gauss map satisfies $\Delta\nu=2\nu$. Thus, $M$ has 1-type pseudo-spherical Gauss map and is not harmonic.
\end{proof}

\begin{theorem}
Let ${\bf x}: (M,g) \longrightarrow\mathbb{S}^{m}_s(1)\subset\mathbb{E}^{m+1}_s$ be a non-totally geodesic isometric immersion of a spacelike surface $M$ in the pseudo-sphere $\mathbb{S}^{m}_s(1)$. The pseudo-spherical Gauss map of the immersion is harmonic if and only if there exists a local isothermal coordinate system $\{u,v\}$ on $M$ such that $g=\mu^2(du^2+dv^2)$ and ${\bf x}$ is a solution of the system
\begin{align}
\begin{split}
\label{diffeq}
{\bf x}_{uu}&=\frac{\mu_u}{\mu}{\bf x}_u-\frac{\mu_v}{\mu}{\bf x}_v-\mu^2{\bf x}+c,\\
{\bf x}_{uv}&=\frac{\mu_v}{\mu}{\bf x}_u+\frac{\mu_u}{\mu}{\bf x}_v,\\
{\bf x}_{vv}&=-\frac{\mu_u}{\mu}{\bf x}_u+\frac{\mu_v}{\mu}{\bf x}_v-\mu^2{\bf x}-c,
\end{split}
\end{align}
where $\mu$ satisfies the equation $(\ln{\mu})_{uu}+(\ln{\mu})_{vv}=-\mu^2$ and 
$c$ is a fixed lightlike vector in $\mathbb{E}^{m+1}_s$.
\end{theorem}

\begin{remark}
The equation for $\mu$ in the theorem can be rewritten as $\Delta(\ln\mu)=-1$, where $\Delta$ is the Laplacian of the surface $(M,g)$ or as $\Delta_0(\ln\mu)=-\mu^2$, where $\Delta_0$ is the Euclidean Laplacian in dimension 2. This equation is known as \emph{Liouville's equation} and characterizes the conformal factor of a surface of constant Gaussian curvature 1.
\end{remark}

\begin{proof}
Assume that ${\bf x}: M\longrightarrow\mathbb{S}^{m}_s(1)\subset\mathbb{E}^{m+1}_s$ is a non-totally geodesic isometric immersion of a spacelike surface $M$ in $\mathbb{S}^{m}_s(1)$ with a harmonic pseudo-spherical Gauss map. Theorem \ref{prop-1} implies that $M$ has a zero mean curvature vector and flat normal connection in $\mathbb{S}^{m}_s(1)$ and that $K=1$. As in the proof of Theorem \ref{1typessurfacearbit} we can choose a local orthonormal frame $\{e_1,e_2\}$ on $M$ such that $h(e_1,e_2)=0$ and $h(e_1,e_1)=-h(e_2,e_2)=\xi$ for a lightlike vector $\xi$. Define the functions $\alpha=\omega_{12}(e_1)$ and $\beta=-\omega_{12}(e_2)$. Then, the Levi-Civita connection of $M$ is determined by 
$$ \nabla_{e_1}e_1=\alpha e_2, \qquad \nabla_{e_1}e_2=-\alpha e_1, \qquad \nabla_{e_2}e_1=-\beta e_2, \qquad \nabla_{e_2}e_2=\beta e_1. $$
Computing $K$ from these we obtain 
\begin{equation}
\label{gaussM}
e_1(\beta)+e_2(\alpha)-\alpha^2-\beta^2=1.
\end{equation}
For $X=Z=e_1$ and $Y=e_2$, respectively $X=e_1$ and $Y=Z= e_2$, the Codazzi equation \eqref{Codazzieq} gives 
\begin{equation}
\label{Dtheo4}
D_{e_2}\xi=2\alpha\xi, \qquad D_{e_1}\xi=2\beta\xi.
\end{equation} 
Using \eqref{Dtheo4}, the fact that $R^D=0$ is equivalent to $e_1(\alpha)=e_2(\beta)$, which is precisely the integrability condition for the system
\begin{equation}
\label{intcondalpha}
e_1(\ln\mu)=-\beta, \qquad e_2(\ln\mu)=-\alpha.
\end{equation}
If $\mu$ is a solution to this system, one can easily check that $[\mu e_1,\mu e_2]=0$, i.e., there exist local coordinates $\{u,v\}$ on $M$ with $\partial_u=\mu e_1$ and $\partial_v=\mu e_2$. In these coordinates the metric reduces to $g=\mu^2(du^2+dv^2)$ and equation \eqref{gaussM} is equivalent to $(\ln{\mu})_{uu}+(\ln{\mu})_{vv}=-\mu^2$. 

Denoting by $\xi_u$ the Euclidean derivative of the vector field $\xi$ in the direction of $\partial_u$, the formula of Gauss for the immersion $\mathbb S^m_s(1) \to \mathbb E^{m+1}_s$ and the formula of Weingarten for the immersion $M \to \mathbb S^m_s(1)$ give $\xi_u = -A_{\xi}\partial_u + D_{\partial_u}\xi - \langle \partial_u,\xi \rangle \mathbf x$. Using equation \eqref{eq:1} and the expressions for $h$, we get $A_\xi e_1=0$ and hence also $A_{\xi}\partial_u=0$. Since also $\langle \partial_u,\xi \rangle = 0$, we find from \eqref{Dtheo4} that $\xi_u=2\beta\mu\xi$. Analogously, one can prove that $\xi_v=2\alpha\mu\xi$. Hence, $\xi$ is a solution to the system
$$ \xi_u=2\beta\mu\xi, \qquad \xi_v=2\alpha\mu\xi. $$
Using the fact that $\partial_u(\ln\mu)=-\beta\mu$ and $\partial_v(\ln\mu)=-\alpha\mu$ from \eqref{intcondalpha}, 
we obtain $ \xi=\frac{c}{\mu^2}$ 
where $c$ is a constant null vector. 

With respect to the coordinate vector fields, the Levi-Civita connection and
the second fundamental form of $M$ in $\mathbb{S}^m_s(1)$ are then given by
\begin{align}
\label{connuv}
&\nabla_{\partial_u}{\partial_u}=\frac{1}{\mu}\left(\mu_u\partial_u-\mu_v\partial_v\right), 
&&\nabla_{\partial_v}{\partial_u}=\frac{1}{\mu}\left(\mu_v\partial_u+\mu_u\partial_v\right),\\
&\nabla_{\partial_v}{\partial_v}=\frac{1}{\mu}\left(\mu_v\partial_v-\mu_u\partial_u\right),\qquad
&&h(\partial_u,\partial_u)=-h(\partial_v,\partial_v)=c, \ h(\partial_u,\partial_v)=0.
\end{align}  
From the Gauss formula, both for the immersion $\mathbb S^m_s(1) \to \mathbb E^{m+1}_s$ and the immersion 
$M \to \mathbb S^m_s(1)$, we get the system \eqref{diffeq}. 

By a straightforward calculation, the converse is obtained.
\end{proof}

\section{Classification results for Lorentzian surfaces}

In this section, we classify the Lorentzian surfaces in $\mathbb{S}^m_s(1)$ whose pseudo-spherical Gauss map is of $1$-type, first in the non-harmonic case, then in the harmonic case. We will need the following classification of flat Lorentzian surfaces with zero mean curvature vector in $\mathbb S^3_1(1)$ and $\mathbb S^3_2(1)$.

\begin{lemma}
\label{lem:classfm}
A flat Lorentzian surface with zero mean curvature vector in $\mathbb{S}^3_s(1) \subset \mathbb E^4_s$ is congruent to an open part of the image of 
\begin{itemize}
\item[(i)] 
$\mathbf{x}(u,v) = \displaystyle\frac{1}{\sqrt{2}}(\cos u, \sin u, \cosh v, \sinh v)$ if $s=1$;
\item[(ii)]
$\mathbf{x}(u,v) = (\cosh u\cos v, \cosh u\sin v, \sinh u\cos v, \sinh u\sin v)$ if $s=2$.
\end{itemize}
\end{lemma}

\begin{remark}
The first surface in Lemma \ref{lem:classfm} is a Lorentzian version of the Clifford torus. The second one is the tensor product of the unit circles $c_1(u)=(\cosh u, \sinh u)$ in the Minkowski plane $\mathbb E^2_1$ and $c_2(v)=(\cos v, \sin v)$ in the Euclidean plane $\mathbb E^2$. It appears for example in the classification of flat tensor products of plane curves in \cite{GV}. After a standard identification of $\mathbb E^4_2$ with $\mathbb C^2$, it can also be seen as the complex curve $c(z)=(\cos z,\sin z)$. A direct calculation shows that the pseudo-spherical Gauss map $\nu$ of any of the two immersions satisfies $\Delta\nu=2\nu$, i.e., that it is of 1-type.
\end{remark}

\begin{proof}
For $s=1$, the proof can be found in \cite{gorokh}. 

Now suppose that ${\bf x}: M_1\longrightarrow\mathbb{S}^3_2(1)$ is an isometric immersion of a flat Lorentzian surface with zero mean curvature vector. Choose pseudo-Euclidean coordinates $\{u,v\}$ on $M_1$ such that $\langle\partial_u,\partial_u\rangle=\langle\partial_v,\partial_v\rangle=0$ and $\langle\partial_u,\partial_v\rangle=-1$. Since $H=-h(\partial_u,\partial_v)$, we obtain $h(\partial_u,\partial_v)=0$. Also, using the Gauss equation \eqref{Gausseq}, flatness implies that $\langle h(\partial_u,\partial_u),h(\partial_v,\partial_v)\rangle=1$. Thus, there exist a unit timelike vector field $\xi$ and a function $a$ on $M_1$ with $a\neq 0$ such that $h(\partial_u,\partial_u)=a\xi$ and $h(\partial_v,\partial_v)=-\frac{1}{a}\xi$. By changing the orientation of $\xi$ if necessary, we may assume that $a>0$. For $X=\partial_u$, $Y=Z=\partial_v$ and $X=\partial_v$, $Y=Z=\partial_u$ in the Codazzi equation \eqref{Codazzieq}, we find $a_u=a_v=0$, which means that $a$ is a positive constant. 

From the Gauss formula, both for the immersions $M_1 \longrightarrow \mathbb S^3_2(1)$ and $\mathbb S^3_2(1) \longrightarrow \mathbb E^4_2$, we obtain the following system of differential equations:
\begin{equation}
\label{syspro}
{\bf x}_{uu}=a\xi, \qquad {\bf x}_{uv}={\bf x}, \qquad {\bf x}_{vv}=-\frac{1}{a}\xi.
\end{equation}
Also, by using the Gauss formula for $\mathbb S^3_2(1) \longrightarrow \mathbb E^4_2$ and the Weingarten formula for $M_1 \longrightarrow \mathbb S^3_2(1)$, we get $\xi_u=-a{\bf x}_v$ and $\xi_v=\frac{1}{a}{\bf x}_u$. Using these, one can determine the solution of system \eqref{syspro} straightforwardly to be 
\begin{multline*}{\bf x}(u,v) = e^{\alpha u+\frac{v}{2\alpha}} \left( \cos \left(\alpha u-\frac{v}{2\alpha}\right) c_1 + \sin \left(\alpha u-\frac{v}{2\alpha}\right) c_2 \right) \\ +e^{-\left(\alpha u+\frac{v}{2\alpha}\right)} \left( \cos \left(\alpha u-\frac{v}{2\alpha}\right) c_3 + \sin \left(\alpha u-\frac{v}{2\alpha}\right) c_4 \right)
\end{multline*}
for $\alpha = \sqrt{a/2}$ and some constant vectors $c_1,c_2,c_3,c_4 \in \mathbb{E}^4_2$. The conditions $\langle \mathbf x,\mathbf x \rangle = -\langle \mathbf x_u,\mathbf x_v \rangle = 1$ and $\langle \mathbf x_u,\mathbf x_u \rangle = \langle \mathbf x_v,\mathbf x_v \rangle = 0$ reduce to
\begin{align*}
& \langle c_1,c_1 \rangle = \langle c_2,c_2 \rangle = \langle c_3,c_3 \rangle = \langle c_4,c_4 \rangle = 0, \\
& \langle c_1,c_2 \rangle = \langle c_1,c_4 \rangle = \langle c_2,c_3 \rangle = \langle c_3,c_4 \rangle = 0, \\
& \langle c_1,c_3 \rangle = \langle c_2,c_4 \rangle = \frac{1}{2}.
\end{align*}
After an isometry of $\mathbb S^3_2(1)$, i.e., an isometry of $\mathbb E^4_2$ leaving $\mathbb S^3_2(1)$ globally invariant, we may assume
$$ c_1 = \frac 12 (1,0,1,0), \quad c_2 = \frac 12 (0,1,0,1), \quad c_3 = \frac 12 (1,0,-1,0), \quad c_4 = \frac 12 (0,1,0,-1).$$
The reparametrization $\alpha u + \frac{v}{2\alpha} \mapsto u$, $\alpha u -\frac{v}{2\alpha} \mapsto v$ now gives the result.
\end{proof}

\begin{theorem}
\label{1typelsurfaceinarbit}
A Lorentzian surface in $\mathbb{S}^{m}_s(1)$ has non-harmonic 1-type pseudo-spherical Gauss map if and only if it is congruent to an open part of the image of
$$ \mathbf{x}(u,v) = \displaystyle\frac{1}{\sqrt{2}}(\cos u, \sin u, \cosh v, \sinh v) $$
in a totally geodesic $\mathbb{S}^3_1(1)\subset\mathbb{S}^{m}_{s}(1)$, or of 
$$ \mathbf{x}(u,v) = (\cosh u\cos v, \cosh u\sin v, \sinh u\cos v, \sinh u\sin v) $$
in a totally geodesic $\mathbb S^3_2(1)\subset\mathbb{S}^m_s(1)$.
\end{theorem}

\begin{proof}
Assume that $M_1$ is a Lorentzian surface in $\mathbb{S}^{m}_s(1)$ with non-harmonic 1-type pseudo-spherical Gauss map. Choose $\{e_1, e_2\}$ to be a local tangent frame satisfying $\langle e_1,e_1\rangle=\langle e_2,e_2\rangle=0$ and $\langle e_1, e_2\rangle=-1$. Then, there exists functions $\alpha$ and $\beta$ on $M_1$ such that
\begin{align}
\label{connM1}
\nabla_{e_1}e_1=\alpha e_1, \quad \nabla_{e_1}e_2=-\alpha e_2, \quad 
\nabla_{e_2}e_1=\beta e_1, \quad \nabla_{e_2}e_2=-\beta e_2.
\end{align}
From Theorem \ref{1typesubmanifold}, $M_1$ has zero mean curvature vector and flat normal connection in $\mathbb{S}^{m}_s(1)$, and constant scalar curvature. Since $H=-h(e_1,e_2)$, we find $h(e_1,e_2)=0$. Also, from the Gauss equation \eqref{Gausseq} the constant Gaussian curvature of $M_1$ is given by $K=1-\langle h(e_1,e_1),h(e_2,e_2) \rangle$. This implies that $\langle h(e_1,e_1),h(e_2,e_2) \rangle = c$ is constant and since the pseudo-spherical Gauss map is not harmonic, we know from Theorem \ref{prop-1} that $c \neq 0$. Put $h(e_1,e_1)=\xi_1$ and $h(e_2,e_2)=\xi_2$. Since $M_1$ has flat normal connection, from the Ricci equation \eqref{Riccieq} we get that 
\begin{equation}
\label{flatnorlor}
\langle\xi_1,\xi\rangle\langle\xi_2,\mu\rangle=\langle\xi_1,\eta\rangle\langle\xi_2,\xi\rangle
\end{equation}
for all normal vector fields $\xi$ and $\eta$, which implies that $\xi_1$ is parallel to $\xi_2$. To see this, let $\{e_3,\dots, e_m\}$ be an orthonormal basis of the normal space of $M_1$ in $\mathbb{S}^m_s(1)$ with $\langle e_i,e_j\rangle=\varepsilon_i\delta_{ij}$, $\varepsilon_i \in \{1,-1\}$. We can express $\xi_1$ and $\xi_2$ in this basis as
$$ \xi_1=\sum_{i=3}^m\varepsilon_i\langle\xi_1,e_i\rangle e_i, \qquad \xi_2=\sum_{i=3}^m\varepsilon_i\langle\xi_2,e_i\rangle e_i. $$ 
For $\xi=e_i$ and $\mu=e_j$ in \eqref{flatnorlor}, we have $\langle\xi_1, e_i\rangle\langle\xi_2, e_j\rangle=\langle\xi_1, e_j\rangle\langle\xi_2, e_i\rangle$ for all $i,j=3, \dots,m$. Thus, the rank of matrix with in its columns the coefficients of $\xi_1$ and $\xi_2$ is less than equal to~$1$, that is, $\xi_1$ and $\xi_2$ are proportional to each other. Also, for $\xi=\xi_1$ and $\eta=\xi_2$ in \eqref{flatnorlor} we get $\langle\xi_1,\xi_1\rangle\langle\xi_2,\xi_2\rangle=c^2>0$, which implies that $\xi_1$ and $\xi_2$ are either both spacelike or both timelike. Hence, there exist a unit normal vector field $\xi'$ and a function $a$ on $M_1$, with $a>0$, such that 
\begin{equation} \label{eqxi1xi2}
\xi_1=a\xi', \qquad \xi_2=\varepsilon\frac{c}{a}\xi',
\end{equation} 
where $\varepsilon=\langle\xi',\xi'\rangle \in \{-1,1\}$. For $X=e_1$ and $Y=Z=e_2$, respectively $X=e_2$ and $Y=Z=e_1$, in the Codazzi equation \eqref{Codazzieq}, we obtain 
\begin{equation} \label{Dxi1xi2}
D_{e_1}\xi_2=-2\alpha\xi_2, \qquad D_{e_2}\xi_1=2\beta\xi_1.
\end{equation} On the other hand, we find from \eqref{eqxi1xi2}
\begin{equation} \label{Dxi1xi2bis}
D_{e_2}\xi_1=e_2(a)\xi'+aD_{e_2}\xi', \qquad
D_{e_1}\xi_2=-\varepsilon \frac{ce_1(a)}{a^2}\xi' + \varepsilon\frac{c}{a}D_{e_1}\xi'.
\end{equation}
Since $\xi'$ is a unit normal vector field, $D_{e_1}\xi'$ and $D_{e_1}\xi'$ are perpendicular to $\xi'$. Thus, by comparing \eqref{Dxi1xi2} and \eqref{Dxi1xi2bis} and using \eqref{eqxi1xi2}, we find $D_{e_1}\xi' = D_{e_2}\xi' =0 $ and 
\begin{equation}
\label{eq1}
e_1(\ln a)=2\alpha, \qquad e_2(\ln a)=2\beta.
\end{equation}
We conclude that $\mbox{Im}\,h=\mbox{Span}\{\xi'\}$ is parallel in the normal bundle and from the reduction theorem by Erbarcher-Magid \cite{erb,M}, $M_1$ is contained in a totally geodesic $\mathbb{S}^3_1(1)$ or $\mathbb{S}^3_2(1)$ of $\mathbb{S}^{m}_s(1)$ depending on whether $\varepsilon=1$ or $\varepsilon = -1$. The compatibility condition for system \eqref{eq1} reads $e_1(\beta)-e_2(\alpha)+2\alpha\beta=0$, which, by \eqref{connM1}, is equivalent to $K=0$. Lemma \ref{lem:classfm} now gives the result.
 
The converse can be straightforwardly verified.
\end{proof}

%\begin{remark}
%In the above proof, when $\mu=0$, i.e., $K=1$, from Proposition \ref{prop-1} it can be easily seen that 
%$\tilde\nu$ is harmonic and $\langle\hat{h}(e_1,e_1),h(e_2,e_2)\rangle=0$.  
%Put $\hat{h}(e_1,e_1)=\xi$ and $\hat{h}(e_2,e_2)=\eta$ for 
%spacelike vectors $\xi$ and $\eta$.
%In this case, $R^D(e_1,e_2;\xi,\eta)=\langle\xi,\xi\rangle\langle\eta,\eta\rangle=0$.
%Thus, there are two possibilities occur:
%\begin{itemize}
%\item $\xi=\eta=0$, that is, $M_1$ is an open portion of totally geodesic %$\mathbb{S}^2_1(1)\subset\mathbb{S}^{m-1}_1(1)$,
%\item $\xi\neq 0$ and $\eta=0$ or $\xi=0$ and $\eta\neq 0$.  
%\end{itemize} 
%\end{remark}

%B.-Y. Chen gave a classification theorem of minimal Lorentzian surface with constant curvature one 
%in the pseudo-sphere $\mathbb{S}^{m-1}_s(1)\subset\mathbb{E}^m_s$. 
%From Theorem 5.1 in \cite{C5} and above remark, the following theorem can be given. 

\begin{theorem}
\label{harmoniclorent}
A non-totally geodesic Lorentzian surface $M_1$ in $\mathbb{S}^{m}_s(1)$ has harmonic pseudo spherical-Gauss map if and only if the immersion $\mathbf{x}: M_1 \longrightarrow \mathbb{S}^{m}_s(1) \subset \mathbb{E}^{m+1}_s$ is given by
\begin{equation}
\mathbf{x}(u,v)=\frac{z(u)}{u+v}-\frac{z'(u)}{2},
\end{equation}
where $z$ is a spacelike curve with constant speed $2$ in the light cone $\{v\in\mathbb{E}^{m+1}_1\ |\ \langle v, v\rangle=0 \}$, satisfying 
$\langle z'', z''\rangle=0$ and $z'''\neq 0$.
\end{theorem}

\begin{proof}
By Theorem \ref{prop-1}, the surface $M_1$ is a minimal surface of constant Gaussian curvature~$1$ with flat normal connection in $\mathbb{S}^{m}_s(1)$. Minimal Lorentzian surfaces with constant Gaussian curvature~$1$ were classified by Chen in Theorem 5.1 of \cite{C5}. The ones which are not totally geodesic but have flat normal connection are exactly those described in the theorem. 
%\red{Do these surfaces have flat normal connection?}
\end{proof}

\begin{example}
We consider a spacelike curve 
\begin{equation*}
z(u)=(\cos(\sqrt{2}u),\sin(\sqrt{2}u), \sinh(\sqrt{2}u), \cosh(\sqrt{2}u))
\end{equation*}
lying the light cone of $\mathbb{E}^4_1$. Then, it can be shown that $\langle z',z'\rangle=4$, $\langle z'', z''\rangle=0$ and $z'''\neq 0$. Thus, from Theorem \ref{harmoniclorent} the Lorentzian surface $M_1$ in $\mathbb S^3_1(1)$ given by the position vector $\mathbf{x}$ has harmonic pseudo-spherical Gauss map.
\end{example}

\end{document}